\documentclass {article}

\usepackage{amsfonts,amssymb,amsmath,mathrsfs,amsthm} 
\usepackage{hyperref}
\usepackage[numbers]{natbib}
\usepackage{graphicx}
\usepackage[utf8]{inputenc}
\usepackage{epstopdf}

\theoremstyle{definition}
\newtheorem{definition}{Definition}[section]

\theoremstyle{plain}
\newtheorem{theorem}{Theorem}[section]

\theoremstyle{remark}

\begin{document}

\title{Uniform persistence in a prey-predator model with a diseased predator}
\author{Tobia Dond\`e}
\date{}

\maketitle

\begin{abstract}
Following the well-extablished mathematical approach to persistence and its developments contained in \cite{BaMaRe14} we give a rigorous theoretical explanation to the numerical results obtained in \cite{BaHi13} on a certain prey-predator model with functional response of Holling type II equipped with an infectious disease in the predator population. 

The proof relies on some repelling conditions that can be applied in an iterative way on a suitable decomposition of the boundary. A full stability analysis is developed, showing how the ``invasion condition'' for the disease is derived. Some counterexamples and possible further investigations are discussed.


\end{abstract}

\section{Introduction}
In \cite{BaHi13} the Authors introduce a model displaying an interesting asymptotic behaviour with peculiar consequences on the basic reproduction number and its reliability in predicting the evolution of the system. Among others, they study a system featuring a prey and a predator population with an infectious disease affecting the latter. It is shown by numerical means that under some constraints on the parameters the basic reproduction number associated with the non trivial stationary equilibrium in the disease-free case fails to be a marker for the spread of the disease. Indeed, as under those hypotheses the equilibrium is not stable, it seems natural to assume that an analysis focused on such point may not reflect the asymptotic dynamics of the system whereas the stable limit cycle that bifurcates from the unstable equilibrium may be the right place to investigate for the long-term behaviour. However, to the best of our knowledge no formal explanation for this insight has been proposed yet, and that is the aim of this work. 

We make use of the mathematical theory of persistence to reach a result that includes and justifies the numerical evidence of \cite{BaHi13}. To this goal we develop a full stability analysis of the model, which is of help in describing the dynamics on the boundary. We consequently find that the repelling conditions that assure the global uniform persistence of the system reduce to an improved ``invasion condition'' for the disease: this is quite interesting in light of the large amount of literature regarding the basic reproduction number $R_0$ and its role as a threshold value for disease spreading.

Nowadays, $R_0$ is a important tool in studying biological and ecological models. Introduced in the well-known Kermack-McKendrick model, it has found many applications and has been generalised in many ways, one for all \cite{DiHeMe90} (for a historical account see also \cite{Ba11}). The easiest interpretation calls $R_0$ the number of secondary cases produced by a infectious individual in a completely susceptible population. When the model displays periodic coefficients the meaning of $R_0$ features in some cases an average quantity over the period, in other cases it encloses more specific informations on the model \cite{BaAi12}. The system we take into account is autonomous but the disease-free dynamics is ruled by a limit cycle, which acts much like a periodic perturbation. Under this point of view the quantity $\overline{R_0}$ evaluated over the period of the limit cycle is close to the one introduced in \cite{BaGu06} and \cite{WaZh08}. 

The role of $R_0$ as an invasion marker motivates its appearances in some survival conditions such as persistence, which is itself a common concept in ecological models. The core idea is the coexistence of all the featuring species and the aim is to give conditions to avoid the extinction of any of them (for interesting and various applications along with some general results on the side of biology and epidemics see \cite{BaMaRe12} and \cite{GaRe16}). However, persistence comes to be a weak assumption when looking at the asymptotic dynamics: for such it is commonly taken into account its stronger counterpart, uniform persistence, which bounds dynamics away from the boundary.

Historically there have been two main mathematical approaches to persistence, with a peak of works at the end of the 80s. Some authors focus on the study of the flow on the boundary, obtaining some acyclic conditions involving chains of suitable decompositions of the boundary (see for instance \cite{BuFrWa86}) while other authors' key feature is a Lyapunov-like function that guarantees the repulsiveness of the boundary (see \cite{Hu84} and \cite{Fo88}). The two approaches find some sort of unification in \cite{Ho89}, being combined through the notion of Morse decomposition of the boundary. For a thorough account on these studies on persistence see \cite{Wa91}. We point out that concepts like uniform persistence or permanence can be defined in quite general frameworks, as done in \cite{FoGi15}, but for the applications sake our theoretical settings will be restricted to locally compact metric spaces.

This work deals with one of the two main models proposed in \cite{BaHi13}. After a theoretical preamble we analyse the model and its flow on the boundary and test the conditions for uniform persistence mainly given in \cite{BaMaRe14}. We discuss the result, which includes and fully supports the evidence reported in \cite{BaHi13}, pointing out also some differences and counterexamples as well as some open problems. The Appendixes deal with the disease-free model, for which existence and stability of the equilibria are illustrated, and a detailed stability analysis of the full model.

\section{Preliminaries}\label{sec:prel}
Let $(\mathcal{X},d)$ be a locally compact metric space and let $\pi$ be a semidynamical system defined on a closed subset $X\subset\mathcal{X}$, i.e. a continuous map
\begin{displaymath}
\pi:X\times\mathbb{R}_+\to X
\end{displaymath}
such that for all $x\in X$ and $t,s\in\mathbb{R}_+=[0,+\infty[$ the following hold:
\begin{align*}
&\pi(x,0)=x\,,\\
&\pi(\pi(x,t),s)=\pi(x,t+s).
\end{align*}
Introducing the $\omega$-limit of a point $x\in X$
\begin{displaymath}
\omega(x):=\{y\in X:\ \exists\, 0<t_n\nearrow\infty\textrm{ such that }\pi(x,t_n)\to y\}
\end{displaymath}
we will assume that $\pi$ is \emph{dissipative}, i.e. for each $x\in X$ it holds $\omega(x)\neq\emptyset$ and $\bigcup_{x\in X}\omega(x)$ has compact closure in $X$.

We fix some notions that will be used throughout the paper.
\begin{definition}\label{definitions}
A subset $Y$ of $X$ is
\begin{itemize}
\item[$\circ$] \emph{forward invariant} with respect to $\pi$ if and only if
\begin{displaymath}
\pi(Y,t)\subseteq Y\quad \forall\, t\geq0;
\end{displaymath}
\item[$\circ$] \emph{isolated} in the sense of \cite{Co78} if it is forward invariant and there exists a neighbourhood $\mathcal{N}(Y)$ of $Y$ such that $Y$ is the maximal invariant set in $\mathcal{N}(Y)$;
\item[$\circ$] \emph{uniform repeller} with respect to $\pi$ if $Y$ is compact and if it exists $\eta>0$ such that
\begin{displaymath}
\liminf_{t\to+\infty}\,d(Y,\pi(x,t))\geq\eta\qquad\forall\, x\in X\setminus Y.
\end{displaymath}
\end{itemize}
The semidynamical system $\pi$ is called \emph{uniformly persistent} if $\partial X$ is a uniform repeller. If $\pi$ is also dissipative then there exists a compact global attractor in the interior of $X$ and $\pi$ is said to be \emph{permanent}.
\end{definition}

We make use of two results coming from the two approaches to persistence. Regarding the ``flow on the boundary'' point of view, a key concept introduced in \cite{Ho89} is to check some necessary and sufficient conditions for uniform repulsiveness on a Morse decomposition of the set.
\begin{definition}\label{morse}
Let $M$ be a closed subset of $X$. A \emph{Morse decomposition} for $M$ is a finite collection $\{\mathcal{M}_1,\mathcal{M}_2,\ldots,\mathcal{M}_n\}$ of pairwise disjoint, compact, forward invariant sets such that for each $x\in M$ one of the following holds:
\begin{itemize}
\item[a)] $x\in \mathcal{M}_i$ for some $i=1,2,\ldots,n$
\item[b)] $\omega(x)\subseteq \mathcal{M}_i$ for some $i=1,2,\ldots,n$.
\end{itemize}
\end{definition}
Recalling the definition of stable manifold of a subset $Y$ of $X$
\begin{displaymath}
W^s(Y):=\{x\in X:\ \omega(x)\subseteq Y\}.
\end{displaymath}
the repelling condition that comes from \cite{BuFrWa86} is stated in \cite{Ho89} as follows.
\begin{theorem}\label{repelmorse}
Let $M$ be a closed forward invariant subset of $X$. Given a Morse decomposition of $M$ such that each $\mathcal{M}_i$ is isolated in $X$, $M$ is a repeller if and only if
\begin{displaymath}
W^s(\mathcal{M}_i)\subseteq M\qquad\forall\, i=1,2,\ldots,n.
\end{displaymath}
\end{theorem}

On the other hand, the main result on the ``Lyapunov-like function'' side, first enunciated in \cite{Hu84} and later improved by \cite{Fo88}, is given here in the spirit of \cite{BaMaRe14}, with a slight modification in order to derive suitable conditions for our model.
\begin{theorem}\label{mainthm}
Let $M\subset X$ be closed and such that $X\setminus M$ is forward invariant. If $X$ contains a compact global attractor $K$ and there exists a closed neighbourhood $V$ of $M$ and a continuous function $P:X\to\mathbb{R}_+$ such that
\begin{itemize}
\item[a)] $P(x)=0 \Longleftrightarrow x\in M$
\item[b)] $\forall\, x\in V\setminus M\quad\exists\, t_x>0:\quad P(\pi(x,t_x))>P(x)$
\end{itemize}
then $M$ is a uniform repeller. If in addition $M$ is forward invariant then we can replace condition b) with
\begin{itemize}
\item[b')] $\exists\, \psi:X\to\mathbb{R}$ continuous and bounded below such that
\begin{itemize}
\item[i)] $\dot{P}(x)\geq P(x)\psi(x)\quad\forall\, x\in V\setminus M$
\item[ii)] $\sup_{t>0}\int_0^t\psi(\pi(x,s))ds>0\quad\forall\, x\in\overline{\omega(M)}$.
\end{itemize}
\end{itemize}
\end{theorem}

The derivative that appears in \textit{ii)} is to be read in the sense of differentiation along the orbits, or Dini derivative, see for instance \cite{La76}. In applications, when $P$ is differentiable on a open subset of $\mathbb{R}^N$ containing $X$ and $\pi$ comes from a ODE of the type $x'=f(x)$, we can write
\begin{displaymath}
\dot{P}(x)=\langle D P(x),f(x)\rangle.
\end{displaymath}

\begin{proof}
We show that \textit{b')} implies \textit{b)} when $M$ is forward invariant: the rest of the proof is carried on in \cite{BaMaRe14} within the same framework. For our purpose we use the proof scheme proposed in \cite{Hu84} and \cite{Fo88}.

At first we shall prove that \textit{ii)} holds true for all $x\in M$. Since $\psi$ is continuous, the map $y\mapsto \sup_{t>0}\int_0^t\psi(\pi(y,s))ds$ is lower semicontinuous. Now, $\overline{\omega(M)}$ is compact, being closed and contained in the compact global attractor $K$: condition \textit{b')} and the permanence of sign guarantee that there exist $\delta>0$ and an open neighbourhood $W$ of $\overline{\omega(M)}$ such that 
\begin{displaymath}
\sup_{t>0}\int_0^t\psi(\pi(y,s))ds>\delta\quad\forall\, y\in W.
\end{displaymath}
Fixed $x\in M$, by definition of $\overline{\omega(M)}$ there exists $t_0$ such that $\pi(x,t)\in W$ for all $t\geq t_0$, so let $x_1:=\pi(x,t_0)$. By definition $x_1\in W$ and by the previous inequality there exists $t_1$ such that $\int_0^{t_1}\psi(\pi(x_1,s))ds>\delta$. If we call $x_2:=\pi(x_1,t_1)=\pi(x,t_0+t_1)\in W$ we can iterate the procedure and come at last to
\begin{align*}
\sup_{t>0}\int_0^t\psi(\pi(x,s))ds&\geq\int_0^{t_0}\psi(\pi(x,s))ds+\sum_{k=1}^{n}\int_{t_{k-1}}^{t_{k-1}+t_k}\psi(\pi(x,s))ds\\
&\geq\int_0^{t_0}\psi(\pi(x,s))ds+\sum_{k=1}^{n}\int_{0}^{t_k}\psi(\pi(x_k,s))ds\\
&>\int_0^{t_0}\psi(\pi(x,s))ds+n\delta.
\end{align*}
By choosing $n$ large enough we conclude that \textit{ii)} holds indeed for any $x\in M$.

Condition \textit{ii)} and the continuity of $\psi$ implies that there exists an open neighbourhood $U$ of $M$ such that
\begin{displaymath}
\forall\, x\in U\quad\exists\, t_x>0:\quad\int_0^{t_x}\psi(\pi(x,s))ds>0.
\end{displaymath}
Applying condition \textit{i)} on a closed neighbourhood $V'$ of $M$ contained in $V\cap U$ returns
\begin{displaymath}
0<\int_0^{t_x}\psi(\pi(x,s))ds\leq\int_0^{t_x}\dfrac{\dot{P}(\pi(x,s))}{P(\pi(x,s))}ds=\log\dfrac{P(\pi(x,{t_x}))}{P(x)}
\end{displaymath}
for all $x\in V'\setminus M$, which of course implies \textit{b)}. Note that the second inequality follows from \textit{i)} even in the case of a general Dini derivative thanks to some general differential inequalities, see \cite{LaLe69}.
\end{proof}

\section{Model and flow analysis}\label{sec:mod}
The model proposed in \cite{BaHi13} accounting for the case of a prey-predator population with a diseased predator is described by the following nonlinear au\-ton\-o\-mous system:
\begin{align}\label{bahipred}
\begin{split}
\begin{cases}
&\dot{N}(t)=rN(t)(1-N(t))-\dfrac{N(t)(S(t)+I(t))}{h+N(t)}\\
&\dot{S}(t)=\dfrac{N(t)(S(t)+I(t))}{h+N(t)}-mS(t)-\beta S(t)I(t)\\
&\dot{I}(t)=\beta S(t)I(t)-(m+\mu)I(t).
\end{cases}
\end{split}
\end{align}
$N(t)$ is the prey population, $S(t)$ is the susceptible predator population and $I(t)$ the infected predator one at time $t\geq0$. The constants $r,h,m,\mu$ and $\beta$ are positive and denote respectively the logistic growth rate of the prey, the half-saturation constant of the Holling type II functional response for the predator, the exponential decay rate of the predator, the mortality increase due to the disease and the disease transmissibility. We will from now on drop the time dependence when no confusion arises.

To fit the model within the theoretical framework of the previous section we set $\mathcal{X}=\mathbb{R}^3$ and $X=\{(x,y,z)\in\mathcal{X}:\,x,y,z\geq0\}=\mathbb{R}_+^3$ is the positive orthant, which is closed. With some abuse of notation let $\pi(x,t):X\times\mathbb{R}_+\to X$ denote the solution of \eqref{bahipred} at time $t$ with initial condition given by $x=(N_0,S_0,I_0)$. To prove that $\pi$ is dissipative we compute the flow against the outward normal vector $(1,1,1)$ to $\{(x,y,z)\in X:\ x+y+z=k\}$ and simple calculations show that the result is negative if
\begin{equation}\label{kbound}
k>1+\dfrac{r}{4m}\,,
\end{equation}
thus once chosen $k$ large enough the flow eventually enters the set 
\begin{displaymath}
X_1:=\{(x,y,z)\in\mathbb{R}_+^3:\ x+y+z\leq k\}.
\end{displaymath}

We now restrict our analysis to $X_1$ and move to the study of the flow on its boundary $\partial X_1$. For convenience sake we split the boundary in three faces (for persistence sake the face $\{(x,y,z)\in X_1:\ x+y+z=k\}$ is of no interest):
\begin{align*}
&\partial_xX_1=\{(x,y,z)\in X_1:\ x=0\}\\
&\partial_yX_1=\{(x,y,z)\in X_1:\ y=0\}\\
&\partial_zX_1=\{(x,y,z)\in X_1:\ z=0\}
\end{align*}
For the detailed study of the flow on $\partial_zX_1$ (disease-free case) we refer to Appendix \ref{sec:A}. As for $\partial_xX_1$ (predator-only case) we have $\dot{N}=0$ and $d(S+I)/dt\leq-m(S+I)$ so the whole face (axes included) is forward invariant and exponentially attracted to the origin. When there are no susceptible predators ($\partial_yX_1$) we have $\dot{S}=(NI)/(h+N)>0$ hence the face (axes excluded) is mapped through $\pi$ into the interior of $X_1$. A scheme of the boundary flow is shown in Figure \ref{Fig1}.
\begin{figure}[ht]
\centering
\includegraphics[width=\textwidth]{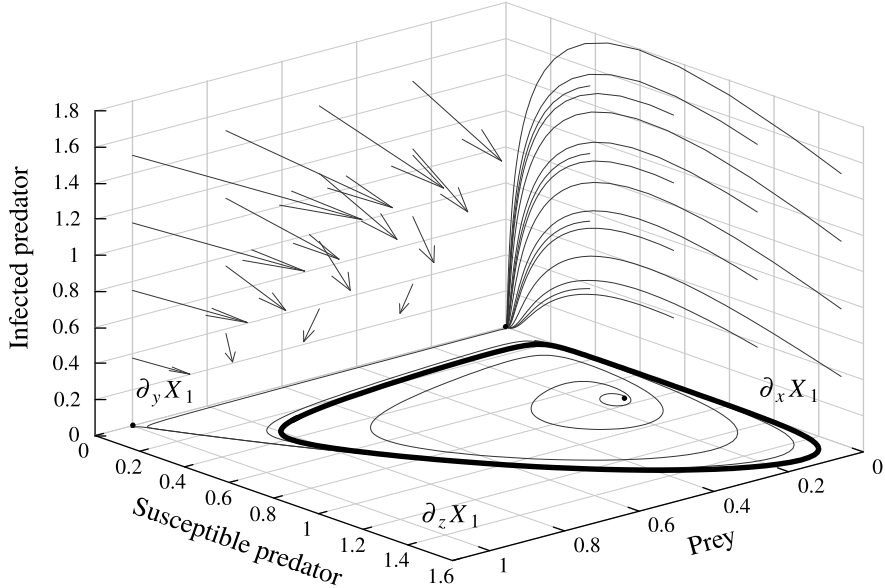}
\caption{The diseased predator model boundary flow. We used the parameters values of \cite{BaHi13}, namely $\mu=0.5,r=2,h=m=0.3$ and $\beta=1.3$. Orbits are displayed for the forward invariant faces $\partial_xX_1$ and $\partial_zX_1$ while for $\partial_yX_1$ we plotted the flow vectors. Equilibria are highlighted in bold.}
\label{Fig1}
\end{figure}

The equilibria lying on the boundary are the origin, the prey-only logistic equilibrium $(1,0,0)$, the disease-free non trivial equilibrium $(N^*, S^*,0)$ and, under the crucial hypothesis
\begin{equation}\label{em}
m<\frac{1-h}{1+h},
\end{equation}
also a limit cycle $\gamma^*$ lying in the disease-free face $\partial_zX_1$ (see Appendix \ref{sec:A}). Appendix \ref{sec:B} is appointed to their stability analysis.

\section{Main result}\label{sec:res}
In this Section we reach the conditions needed for the global uniform persistence of system \eqref{bahipred} by proving repeatedly the repulsiveness of elements of a suitable decomposition of $\partial X_1$, removing them one by one and restricting the analysis to the remaining part of the boundary. The need for this iterative procedure will be explained in Section \ref{sec:dis}.

Let
\begin{displaymath}
M_1:=\{(x,y,z)\in X_1:\ x=0\}.
\end{displaymath}
$M_1$ is a closed and invariant subset of $X_1$. A Morse decomposition for $M_1$ is given by $\mathcal{M}_1=\{(0,0,0)\}$, which is indeed a Morse set and $\omega(M_1)=\mathcal{M}_1$. Now, $W^s(\mathcal{M}_1)=M_1$ (refer to Appendix \ref{sec:B} for the stability analysis): the saddle behaviour of the origin guarantees that it is isolated in $X_1$ and we can hence apply Theorem \ref{repelmorse} to obtain the uniform repulsiveness of $M_1$ in $X_1$. This in turns implies the permanence of $\pi$ on $X_1\setminus M_1$, that is the existence of a compact invariant attractor $K_1\subset X_1\setminus M_1$ such that $d(K_1,M_1)>0$ (see \cite[Theorem 2.2]{Hu84} or \cite[Chapter II.5]{Co78}). 

We can now analyse our flow $\pi$ on
\begin{displaymath}
X_2:=X_1\setminus M_1
\end{displaymath}
which is an open subset of $X_1$, hence a locally compact metric space. Setting
\begin{displaymath}
M_2:=\{(x,y,z)\in X_2:\ y=z=0\}
\end{displaymath}
we see that $M_2$ is forward invariant and closed with respect to $X_2$. $\mathcal{M}_2=\{(1,0,0)\}$ attracts all the points in $M_2$ and is indeed a Morse decomposition for $M_2$. Again, by the stability analysis $W^s(\mathcal{M}_2)=M_2$ and hence Theorem \ref{repelmorse} holds, giving uniform repulsiveness of $M_2$ and a compact invariant attractor $K_2$ inside $X_2\setminus M_2$ such that $d(K_2,M_2)>0$.

Now set
\begin{displaymath}
X_3:=X_2\setminus M_2\,,
\end{displaymath}
which is again a locally compact metric space. We define
\begin{displaymath}
M_3:=\{(x,y,z)\in X_3:\ y=0\}.
\end{displaymath}
$M_3$ is closed in $X_3$ and $X_3\setminus M_3$ is forward invariant. Note that $M_3$ itself is not forward invariant, thus some of the classical results do not apply. Recall that by the previous step $X_3$ contains a compact global attractor $K_2$. The set
\begin{displaymath}
V:=\left\{(x,y,z)\in X_3:\ y\leq\dfrac{xz}{(h+x)(m+\beta z)}\right\}
\end{displaymath}
is a closed neighbourhood of $M_3$ and on $V$ we define the real-valued function
\begin{displaymath}
P(x,y,z):=y.
\end{displaymath}
By definition $P(x,y,z)=0$ if and only if $(x,y,z)\in M_3$ and if $(x_0,y_0,z_0)\in V\setminus M_3$ then
\begin{align*}
\dot{P}(x_0,y_0,z_0)&=\dfrac{x_0z_0}{h+x_0}+y_0\left(\dfrac{x_0}{h+x_0}-m-\beta z_0\right)\\
&>\dfrac{x_0z_0}{h+x_0}-y_0(m+\beta z_0)\\
&\geq\dfrac{x_0z_0}{h+x_0}-\dfrac{x_0z_0}{(h+x_0)(m+\beta z_0)}(m+\beta z_0)=0
\end{align*}
since we are in $V$. Hence, $P$ increases along orbits that have starting points in $V\setminus M_3$, which means
\begin{displaymath}
\forall\, x\in V\setminus M_3\quad\exists\, t_x>0:\quad P(\pi(x,t_x))>P(x).
\end{displaymath}
We are under the hypotheses \textit{a)} and \textit{b)} of Theorem \ref{mainthm} and we obtain that $M_3$ is a uniform repeller in $X_3$. As before, this leads to the existence of a compact global attractor $K_3\subseteq K_1$ inside $X_3\setminus M_3$ such that $d(K_3,M_3)>0$.

Eventually, call
\begin{displaymath}
X_4:=X_3\setminus M_3
\end{displaymath}
and
\begin{displaymath}
M_4:=\{(x,y,z)\in X_4:\ z=0\}.
\end{displaymath}
$M_4\subset X_4$ is closed and globally invariant thanks to the Kolmogorov structure of the third equation of \eqref{bahipred}: by the uniqueness of the Cauchy problem associated with \eqref{bahipred} this implies that $X_4\setminus M_4$ is forward invariant. By the previous step $X_4$ contains a compact global attractor $K_3$. On $X_4$ we define the non negative function
\begin{displaymath}
P(x,y,z):=z
\end{displaymath}
and we note that $P(x,y,z)=0$ if and only if $(x,y,z)\in M_4$. Setting $V=X_4$ it holds
\begin{displaymath}
\dot{P}(x,y,z)=\dot{I}(t)=I(t)(\beta S(t)-(m+\mu))=P(x,y,z)\psi(x,y,z)
\end{displaymath}
with $\psi(x,y,z)=\beta y-(m+\mu)$. In order to satisfy hypothesis \textit{b')} of Theorem \ref{mainthm} we need to check that
\begin{displaymath}
\sup_{t>0}\int_0^t\psi(\pi(x,s))ds>0\qquad\forall\, x\in\overline{\omega(M_4)}.
\end{displaymath}
Since $\overline{\omega(M_4)}=\mathcal{M}_3\cup \mathcal{M}_4$ where $\mathcal{M}_3=\{(N^*,S^*,0)\}$ is the non trivial stationary equilibrium and $\mathcal{M}_4=\gamma^*$ is the stable limit cycle arising from the disease-free model, in the first case the above condition reads $\beta S^*>m+\mu$ which is equivalent to
\begin{displaymath}
R_0^*>1
\end{displaymath}
where
\begin{displaymath}
R_0^*:=\dfrac{\beta S^*}{m+\mu},
\end{displaymath}
while in the second case we have
\begin{equation}\label{supcycle}\tag{$R_0^{\text{sup}}$}
\forall\,(N_0,S_0,0)\in\gamma^*\qquad\sup_{t>0}\int_0^t (\beta S(s)-(m+\mu))ds>0\quad\text{for }S(0)=S_0.
\end{equation}

We are now able to state our main result.
\begin{theorem}\label{persistencepred}
Let system \eqref{bahipred} be defined on the positive orthant $\mathbb{R}_+^3$ with associated flow $\pi$ and let \eqref{supcycle} hold along with
\begin{displaymath}
m<\dfrac{1-h}{1+h},\qquad R_0^*>1.
\end{displaymath}
Then $\pi$ is uniformly persistent.
\end{theorem}
\begin{proof}
By the previous calculations $\pi$ is dissipative, hence for a fixed $k$ satisfying \eqref{kbound} each orbit $\pi(x,t)$ enters definitively the compact set
\begin{displaymath}
X_1=\{(x,y,z)\in\mathbb{R}_+^3:\ x+y+z\leq k\}.
\end{displaymath}
Condition \eqref{em} guarantees the instability of the stationary equilibrium $(N^*,S^*,0)$ as well as the existence and stability of a limit cycle $\gamma^*$ in $\partial_zX_1$ (note that the lower bound on $k$ guarantees $\gamma^*\subset X_1$). Following the iterative scheme based on Theorems \ref{repelmorse} and \ref{mainthm} illustrated in this Section the two conditions needed for the global uniform persistence of $\pi$ in $X_1$ come to be $R_0^*>1$ and \eqref{supcycle}. Since any orbit entering $X_1$ is thereby trapped and the argument is irrespective of the choice of $k$ if large enough, uniform persistence extends to the whole positive orthant. The proof is complete. 
\end{proof}

\section{Discussion}\label{sec:dis}
Condition \eqref{em} causes a shift of stability in the disease-free plane from the non trivial stationary equilibrium to the limit cycle. The natural thought would be that, since the dynamics is conveyed to the limit cycle instead of the equilibrium point, a repelling condition should be given on $\gamma^*$. This heuristic hypothesis proves indeed true: condition \eqref{supcycle} is the key assumption for uniform persistence, and is given on the points of the limit cycle itself. A further evidence that the focus should be adjusted on the limit cycle instead of the unstable equilibrium $(N^*,S^*,0)$ comes from the flow on the boundary point of view, where the equilibria $\{\mathcal{M}_1,\mathcal{M}_2,\mathcal{M}_3,\mathcal{M}_4\}$ form a Morse decomposition of $\partial X_1$ which is acyclic and ends precisely with the limit cycle:
\begin{displaymath}
\mathcal{M}_1\rightarrow \mathcal{M}_2\rightarrow \mathcal{M}_4 \leftarrow \mathcal{M}_3.
\end{displaymath}

These observations may lead to the conclusion that no condition is needed on $(N^*,S^*,0)$. This proves indeed wrong: in Theorem \ref{persistencepred} we still require $R_0^*>1$ and this hypothesis cannot be avoided in view of the last step of the iterative scheme illustrated above and the stability analysis carried on in the Appendixes. In fact, $R_0^*-1$ is concordant with the third eigenvalue of the Jacobian evaluated in the non trivial equilibrium point (see Appendix \ref{sec:B}): if negative, a one dimensional stable manifold arises that has non empty intersection with the interior of $X_1$, thus meaning that initialising system \eqref{bahipred} with points in int$\,X_1\cap W^s(\{(N^*,S^*,0)\})$ results in the asymptotic extinction of the disease. 

An example of this phenomenon is provided by the system
\begin{align}\label{dufsys}
\begin{split}
\begin{cases}
&\dot{x}=y-\varepsilon\left(\dfrac{x^3}{3}-x\right)\\
&\dot{y}=-x\\
&\dot{z}=\delta\varphi(x^2+y^2)z
\end{cases}
\end{split}
\end{align}
with $\varepsilon,\delta>0$ and $\varphi(s)=\max\{\min\{s-2,0.5\},-1\}$. A van der Pol equation in the Li\'enard plane $(x,y)$ is coupled with a radius-dependent height $z$ which takes into account the two well-known $\omega$-limits of the equation, the unstable origin and a stable limit cycle on the plane that for $0<\varepsilon\ll1$ is close to the circumference $x^2+y^2=4$. The result is shown in Figure \ref{Fig2}. The $z$-axis is the stable one dimensional manifold of the origin: starting arbitrarily close but not on it leads the orbit to approach the $(x,y)$-plane so that the effect of the instability of the origin pushes the orbit towards the limit cycle. The choice of $\varphi$ causes the $z$-component to increase near the limit cycle, so that condition \eqref{supcycle} holds but uniform persistence fails due to the $z$-axis being exponentially attracted to the origin.
\begin{figure}[ht]
\centering
\includegraphics[width=\textwidth]{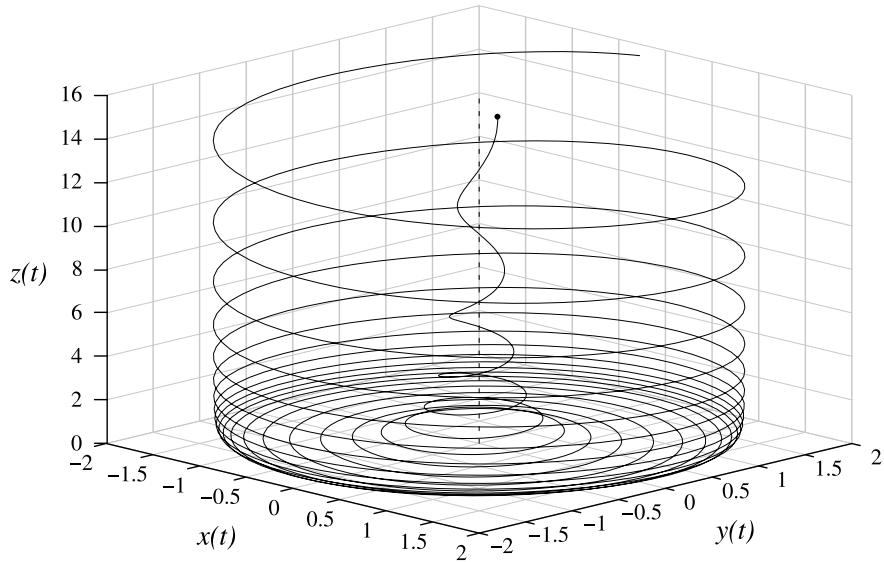}
\caption{An orbit of system \eqref{dufsys} starting close to the $z$-axis, moving away from the unstable origin in proximity of the $(x,y)$-plane and approaching the limit cycle, which is repulsive in the $z$-direction. Parameters: $\varepsilon=\delta=0.1$, $x_0=y_0=0.1$ and $z_0=15$. The time span is $[0,150]\,$.
}
\label{Fig2}
\end{figure}

In analysing numerically the system \eqref{bahipred} Bate and Hilker provide an elegant notion of average basic reproduction number $\overline{R_0}$, based on the limit cycle and similar in structure to $R_0^*$:
\begin{equation}\label{Rnotbar}
\overline{R_0}:=\dfrac{\beta\overline{S}}{m+\mu}\,,
\end{equation}
with $\overline{S}=\int_0^TS(t)dt$ where $T$ is the period of the limit cycle and $S(0)=S(T)\in\gamma^*$. Their condition for the persistence of the disease becomes $\overline{R_0}>1$, which of course implies \eqref{supcycle}. We can thus say that Theorem \ref{persistencepred} is a rigorous theoretical result that supports the numerical evidence provided thereby. We also reach some more degree of generality due to our specific condition on the limit cycle. It is important to stress that although $R_0^*>1$ alone does not guarantee uniform persistence, as shown in \cite{BaHi13} where the disease extinguishes in some situations where $\overline{R_0}<1<R_0^\ast$, it is still a necessary condition in order to avoid the presence of an internal one dimensional stable manifold for the stationary equilibrium. This assumption is not required by Bate and Hilker: of course, in that case persistence holds almost everywhere as only a one dimensional set of initial values leads system \eqref{bahipred} to the asymptotic extinction of the disease; moreover, such set is known theoretically from its tangent line but is hard to compute precisely or to spot by numerical means. Nevertheless it is interesting to note how the necessary condition $R_0^*>1$ arises both from the stability analysis and hypothesis \textit{b')} of Theorem \ref{mainthm}.

An interesting optimality question arises from the comparison between the conditions \eqref{supcycle} and $\overline{R_0}>1$, that is which is the first time $\tau>0$ to realise
\begin{displaymath}
\int_0^\tau\psi(\pi(x,s))ds>0\qquad\forall\, x\in\gamma^*.
\end{displaymath}
For large values of $\beta$ we would expect $\tau<T$ as the limit cycle does not depend on $\beta$ and it is easier for the disease to survive with a high infection rate. However, the dependence on the other parameters is not trivial and may display some interesting behaviours which are difficult to investigate theoretically as little is known about the limit cycle and its properties.

The proof of Theorem \ref{persistencepred} is original in its iterative application of Theorems \ref{repelmorse} and \ref{mainthm} to some carefully chosen decomposition of the boundary, progressively removing those sets for which uniform repulsiveness has been proven. This approach is not redundant in the sense that standard techniques does not apply to such a heterogeneous boundary, which is neither globally forward invariant nor all weakly repulsive. In this setting, the results in \cite{Hu84} and \cite{Fo88} cannot be used on the whole boundary. Even the powerful unifying tool \cite[Corollary 4]{Ho89} fails on the disease-free face as all the equilibria (Morse sets) are involved but $P(x,y,z)=z$ does not satisfy the hypotheses of the Corollary in the origin and in the prey-only equilibrium. To avoid this points we should restrict $P$ to a subset of the disease-free face which is disjoint from the $x$-axis, but then we lose either the forward invariance or the closure of such subset. It is not clear if a different Lyapunov-like function could be defined: it should take into account that each equilibrium has its own instability direction, as shown in Appendix \ref{sec:B}. This is also the reason why Theorem \ref{mainthm} itself cannot be applied on the whole $X_1$. 

We draw the attention on the abstract problem of generalising such an iterative technique. One issue is intrinsic and lies in the choice of the partition and in which order to remove its sets from the boundary. We speculate that this process could follow the acyclic chain of Morse sets, as it is in our case: there are however non trivial issues that could arise, especially when the chain is more involved than the one we showed at the beginning of this Section.
Another question is if Theorem \ref{persistencepred} could be derived following other approaches. System \eqref{bahipred} falls under the class of models (23) of \cite{GaRe16}: we speculate that some of the arguments could be adapted in order to prove the persistence of the disease even if the results therein are given for interactions of competitive or cooperative type, which is not our case.

We conclude by pointing out some possible further extensions of this work. We dealt with one of the two main models proposed in \cite{BaHi13}, the diseased predator one, but some extensions are also thereby analysed, for example the case when both populations share the disease or when the disease alters the density dependence of the infected population. It may be interesting to investigate if the techniques we adopted here can also apply to those models so to reach similar persistence results. Another interesting possibility is to allow the disease transmissibility $\beta$ to vary in time, for example choosing a periodic function describing seasonality: this may link again to the previously cited works \cite{BaMaRe12,GaRe16}.

\appendix
\section{Appendix: Disease-free model}\label{sec:A}
The underlying prey-predator model reads
\begin{align}\label{pred}
\begin{split}
\begin{cases}
&\dot{N}=rN(1-N)-\dfrac{NS}{h+N}=N\left[r(1-N)-\dfrac{S}{h+N}\right]\\
&\dot{S}=\dfrac{NS}{h+N}-mS=S\left[\dfrac{N}{h+N}-m\right]
\end{cases}
\end{split}
\end{align}
where the Kolmogorov-like structure has been highlighted.

We identify three equilibrium points: the origin $(0,0)$, the trivial prey-only equilibrium $(1,0)$ and a non trivial one $(N^*,S^*)$ which is given by
\begin{displaymath}
N^*=\dfrac{mh}{1-m}\qquad S^*=r(1- N^*)(h+ N^*).
\end{displaymath}
The Jacobian of system \eqref{pred} evaluated in the equilibrium points reads
\begin{displaymath}
\mathbf{J}(0,0)=\left(\begin{array}{cc}r&0\\0&-m\end{array}\right),\quad \mathbf{J}(1,0)=\left(\begin{array}{cc}-r&-\dfrac{1}{h+1}\\0&\dfrac{1}{h+1}-m\end{array}\right)
\end{displaymath}
and
\begin{displaymath}
\mathbf{J}( N^*, S^*)=\left(\begin{array}{cc}rm\left(1-\dfrac{1+m}{1-m}h\right)&-m\\r(1-m(1+h))&0\end{array}\right).
\end{displaymath}
The origin is always a saddle. To avoid the stability of the trivial equilibrium it must be
\begin{displaymath}
m<\frac{1}{1+h}
\end{displaymath}
which also ensures 
\begin{displaymath}
\textrm{det}\,\mathbf{J}( N^*, S^*)=mr(1-m(1+h))>0
\end{displaymath}
so that the nontrivial equilibrium is unstable if
\begin{displaymath}
\textrm{tr}\,\mathbf{J}(N^*,S^*)=rm\left(1-\dfrac{1+m}{1-m}h\right)>0\quad\Longleftrightarrow\quad h<\frac{1-m}{1+m}
\end{displaymath}
which is equivalent to \eqref{em}. Three cases are hence given:
\begin{itemize}
\item[$\circ$] $m>1/(1+h)$: the logistic equilibrium is stable
\item[$\circ$] $(1-h)/(1+h)<m<1/(1+h)$: $(1,0)$ is unstable and $(N^*,S^*)$ is stable
\item[$\circ$] $m<(1-h)/(1+h)$: both equilibrium points are unstable.
\end{itemize}
We choose the third hypothesis and prove that a stable and unique limit cycle bifurcates from the unstable equilibrium $(N^*,S^*)$.

First of all note that system \eqref{pred} is dissipative once defined on a set
\begin{displaymath}
X=\{(x,y)\in\mathbb{R}_+^2:\ x+y\leq k\}
\end{displaymath}
with $k$ large enough. We could prove that the flow $\pi$ associated with \eqref{pred} is uniformly persistent by means of Theorem \ref{mainthm} but this comes straightforwardly from \cite[Theorem 3.2(b)]{GaRe16}.

The existence of the limit cycle comes from the Poincar\'e-Bendixson annular region Theorem while its uniqueness is proven in \cite{Ch81}: the stability is a consequence of the uniqueness as shown in \cite{HsHuWa78b} and in the more classical \cite{Le63}. For a general result on the limit cycles of Kolmogorov systems like \eqref{pred} see \cite{YuChDu12}.

\section{Appendix: Stability analysis}\label{sec:B}
The four equilibria on the boundary pointed out in Section \ref{sec:mod} are the origin, the trivial logistic equilibrium $(1,0,0)$, the disease-free equilibrium $( N^*, S^*,0)$ and the limit cycle in the disease-free plane that we named $\gamma^*$. We choose $m$ as in \eqref{em} such that on that plane all critical points are unstable and the limit cycle is stable. We now want to study the stable manifolds of the equilibrium points.

Let us write the linearised matrix for the system \eqref{bahipred}:
\begin{displaymath}
\mathbf{J}(x,y,z)=\left(\begin{array}{ccc}r(1-2x)-\dfrac{h(y+z)}{(h+x)^2}&-\dfrac{x}{h+x}&-\dfrac{x}{h+x}\\\dfrac{h(y+z)}{(h+x)^2}&\dfrac{x}{h+x}-m-\beta z&\dfrac{x}{h+x}-\beta y\\0&\beta z&\beta y-(m+\mu)\end{array}\right).
\end{displaymath}
Evaluating in the equilibrium points:
\begin{displaymath}
\mathbf{J}(0,0,0)=\left(\begin{array}{ccc}r&0&0\\0&-m&0\\0&0&-(m+\mu)\end{array}\right)
\end{displaymath}
\begin{displaymath}
\mathbf{J}(1,0,0)=\left(\begin{array}{ccc}-r&-\dfrac{1}{h+1}&-\dfrac{1}{h+1}\\0&\dfrac{1}{h+1}-m&\dfrac{1}{h+1}\\0&0&-(m+\mu)\end{array}\right)
\end{displaymath}
\begin{displaymath}
\mathbf{J}(N^*,S^*,0)=\left(\begin{array}{ccc}rm\left(1-\dfrac{1+m}{1-m}h\right)&-m&-m\\r(1-m(1+h))&0&m-\beta S^*\\0&0&\beta S^*-(m+\mu)\end{array}\right).
\end{displaymath}
It is easy to see that
\begin{displaymath}
W^s(\{(0,0,0\})\cap X_1=\partial_x X_1=\{(x,y,z)\in X_1:\ x=0\}.
\end{displaymath}
As for $(1,0,0)$, the tangent plane to its stable manifold is given by
\begin{displaymath}
k_1(1,0,0)+k_2\left(\dfrac{\mu}{(1+\mu(h+1))(m+\mu-r)},-\dfrac{1}{1+\mu(h+1)},1\right),\quad k_1,k_2\in\mathbb{R}
\end{displaymath}
which lies strictly outside the positive orthant except for the points in $\{(x,y,z)\in X_1:\ y=z=0\}$, all belonging to the stable manifold of $(1,0,0)$ except for the origin. It holds $W^s(\{(1,0,0)\})\cap\text{int}\,X_1=\emptyset$ because an orbit belonging to $W^s(\{(1,0,0)\})$ should approach $(1,0,0)$ tangentially to the above plane but this is impossible from the inside of $X_1$ as its boundary is either forward invariant ($\partial_z X_1, \partial_x X_1$) or repulsive ($\partial_y X_1$). Thus
\begin{displaymath}
W^s(\{(1,0,0)\})\cap X_1=\{(x,y,z)\in X_1:\ x>0,\ y=z=0\}.
\end{displaymath}

The first two eigenvalues $\lambda_1,\lambda_2$ for the non trivial equilibrium $(N^*,S^*,0)$ are strictly positive, as illustrated in Appendix \ref{sec:A}. $\lambda_3=\beta S^*-(m+\mu)$ could in principle possess a one dimensional stable manifold if
\begin{displaymath}
R_0^*:=\dfrac{\beta S^*}{m+\mu}<1.
\end{displaymath}
The tangent line in $(N^*,S^*,0)$ to this manifold is
\begin{displaymath}
(N^*,S^*,0)+k\left(\dfrac{\mu m}{f(\lambda_1,\lambda_2,\lambda_3)},\dfrac{\mu(\lambda_1+\lambda_2-\lambda_3)}{f(\lambda_1,\lambda_2,\lambda_3)},1\right),\quad k\in\mathbb{R}
\end{displaymath}
where $f(\lambda_1,\lambda_2,\lambda_3)=(\lambda_1+\lambda_2)\lambda_3-(\lambda_1\lambda_2+\lambda_3^2)$. The line cuts through the disease-free face, hence the stable manifold passes through the interior of $X_1$ and can potentially lead to the extinction of the disease. To avoid this chance we ask for
\begin{displaymath}
R_0^*>1.
\end{displaymath}

\section*{Acknowledgements}
The author is truly grateful to prof. Fabio Zanolin without whose supervision this work wouldn't have been accomplished. The author wishes to thank also prof. Carlota Rebelo and prof. A\-les\-san\-dro Mar\-ghe\-ri for the helpful remarks and even more for suggesting the paper \cite{BaHi13} which introduced our main subject of investigation.

Work partially supported by Grup\-po Na\-zio\-na\-le per l'Anali\-si Ma\-te\-ma\-ti\-ca, la Pro\-ba\-bi\-li\-t\`{a} e le lo\-ro Appli\-ca\-zio\-ni (GNAMPA) of Isti\-tu\-to Na\-zio\-na\-le di Al\-ta Ma\-te\-ma\-ti\-ca (INdAM). Progetto di Ricerca 2017: ``Problemi differenziali con peso indefinito: tra metodi topologici e aspetti dinamici''.

\end{document}